\documentclass[12pt,a4paper,reqno]{amsart}
\usepackage{amssymb}
\usepackage{amscd}
\numberwithin{equation}{section}

\theoremstyle{plain}

\newtheorem{theorem}[subsection]{Theorem}
\newtheorem{proposition}[subsection]{Proposition}
\newtheorem{lemma}[subsection]{Lemma}
\newtheorem{corollary}[subsection]{Corollary}

\theoremstyle{definition}

\newtheorem{definition}[subsection]{Definition}
\newtheorem{remark}[subsection]{Remark}

\newcommand\eps{\varepsilon}
\newcommand\R{{\mathbf{R}}}

\newcommand\HH{{\mathbb{H}}}
\newcommand\I{{\mathcal{I}}}
\newcommand\C{{\mathbb{C}}}

\newcommand{\beq}{\begin{equation*}}
\newcommand{\eeq}{\end{equation*}}

\begin{document}

\title{An incidence theorem in higher dimensions}

\author{J\'ozsef Solymosi}
\address{Department of Mathematics, University of British Columbia}
\email{solymosi@math.ubc.ca }
%\thanks{J. Solymosi is supported by an NSERC grant.}

\author{Terence Tao}
\address{UCLA Department of Mathematics, Los Angeles, CA 90095-1555}
\email{tao@math.ucla.edu}
%\thanks{T. Tao is supported by a grant from the MacArthur Foundation, and by NSF grant DMS-0649473 and the NSF %Waterman award.}


\begin{abstract} We prove almost tight bounds on the number of incidences between points and $k$-dimensional varieties of bounded degree in $\R^d$.  Our main tools are the polynomial ham sandwich theorem and induction on both the dimension and the number of points.
\end{abstract}

\maketitle
% \today

\section{Introduction}

Given a collection $P$ of points in some space, and a collection $L$ of sets in that same space, let $I(P,L) := \{ (p,\ell) \in P \times L: p \in \ell \}$ be the set of incidences.  One of the objectives in combinatorial incidence geometry is to obtain good bounds on the cardinality $|I(P,L)|$ on the number of incidences between finite collections $P, L$, subject to various hypotheses on $P$ and $L$.  For instance, we have the classical result of Szemer\'edi and Trotter \cite{SzT}:

\begin{theorem}[Szemer\'edi-Trotter theorem]\label{szt-thm}\cite{SzT}  Let $P$ be a finite set of points in $\R^d$ for some $d \geq 2$, and let $L$ be a finite set of lines in $\R^d$.  Then
\begin{equation}\label{szt-bound}
|I(P,L)| \leq C( |P|^{2/3} |L|^{2/3} + |P| + |L| )
\end{equation}
for some absolute constant $C$.
\end{theorem}

This theorem is usually stated in the two-dimensional setting $d=2$, but the higher-dimensional case is an immediate consequence by applying a generic projection from $\R^d$ to $\R^2$ (see Section \ref{large-sec} for a discussion of this argument).  It is known that this bound is sharp except for the constant $C$; see \cite{SzT}.

Various mathematical problems can be transformed to a question about incidence bounds of Szemer\'edi-Trotter type.  For instance, Elekes \cite{E2} used the above theorem in an unexpected fashion to obtain new bounds on the sum-product problem. In the 1990s, Wolff \cite{wolff} observed that bounds on the number of incidences might be used in problems related to the Kakeya conjecture, one of the central
conjectures in harmonic analysis.  Bennett, Carbery and Tao \cite{BCT} established a connection between multilinear Kakeya estimates and bounds on number of incidences between points and lines in three dimensions. Very recently, Guth and Katz \cite{GK} used bounds on the number of incidences between points and lines in three dimensions as part of their solution to the Erd\H os Distinct Distances Problem. They used an important tool, the so-called polynomial ham sandwich theorem. This theorem will be a crucial part of this paper as well. The applicability of the polynomial ham
sandwich theorem to Szemer\'edi-Trotter type theorems was also recently emphasized in \cite{kaplan}.

Further applications of Szemer\'edi-Trotter type incidence bounds in mathematics and theoretical computer science, as well as several open problems, are discussed in the surveys and books of Elekes \cite{E1}, Sz\'ekely \cite{Szek}, Pach and Sharir \cite{PS}, Brass, Moser, and Pach \cite{BMP}, and Matou\v{s}ek \cite{Ma}.

In the survey \cite{E1}, Elekes listed some nice applications of point-line incidence bounds in  the \emph{complex} plane $\C^2$, where the lines are now \emph{complex} lines (and thus are also real planes). In this paper he referred to a (then) recent result of T\'oth \cite{To} which proves the point-line incidence bound \eqref{szt-bound} in this situation (with a different constant $C$).  Up to this constant multiplier, this bound is optimal. 
Our argument is different from and simpler than the one in T\'oth's paper; however, the bounds in this paper are slightly weaker than those in \cite{To}.

The main goal of our paper is to establish near-sharp Szemer\'edi-Trotter type bounds on the number of incidences between points and $k$-dimensional algebraic varieties in $\R^d$ for various values of $k$ and $d$, under some ``pseudoline'' hypotheses on the algebraic varieties; see Theorem \ref{main} for a precise statement.  In particular, we obtain near-sharp bounds for point-line incidences in $\C^2$, obtaining a ``cheap'' version of the result of T\'oth mentioned previously.

Our argument is based on the ``polynomial method'' as used by Guth and Katz \cite{GK}, combined with an induction on the size of the point set $P$.  The inductive nature of our arguments causes us to lose an arbitrarily small epsilon term in the exponents, but the bounds are otherwise sharp.

As in \cite{GK}, our arguments rely on an efficient cell decomposition provided to us by the polynomial ham sandwich theorem
(see Corollary \ref{cell-decomp}).  However, the key innovation here, as compared to the arguments in \cite{GK}, is that this
decomposition will only be used to partition the point set into a \emph{bounded} number of cells, rather than a large number of cells. (Similar recursive space partitioning techniques were used by Agarwal and Sharir \cite{agshar}.)
This makes the contribution of the cell boundaries much easier to handle (as they come from varieties of bounded degree, rather than large degree).
The price one pays for using this milder cell decomposition is that the contribution of the cell interiors can no longer be handled by ``trivial''
bounds.  However, it turns out that one can use a bound coming from an induction hypothesis as a substitute for the trivial bounds, so long
as one is willing to concede an epsilon factor in the inductive bound.  This technique appears to be quite general, and suggests that one can use
induction to significantly reduce the need for quantitative control of the geometry of high-degree algebraic varieties when applying the polynomial
method to incidence problems, provided that one is willing to lose some epsilons in the final bounds.

Our results have some similarities with existing results in the literature; we discuss these connections in Section \ref{comparison}.

\subsection{Notation}

We use the usual asymptotic notation $X = O(Y)$ or $X \ll Y$ to denote the estimate $X \leq CY$ for some absolute constant $C$.  If we need the implied constant $C$ to depend on additional parameters, we indicate this by subscripts, thus for instance $X = O_d(Y)$ or $X \ll_d Y$ denotes the estimate $X \leq C_d Y$ for some quantity $C_d$ depending on $d$.

\subsection{Acknowledgements}

The authors are very grateful to Boris Bukh, Nets Katz, Jordan Ellenberg,
and Josh Zahl for helpful discussions and to the Isaac Newton Institute, Cambridge for hospitality while this research was being conducted. We also thank Alex Iosevich, Izabella {\L}aba, Ji\v{r}\'{\i} Matou\v{s}ek, J\'anos Pach, and Miguel Walsh for comments, references, and corrections to an earlier draft of this manuscript. We are thankful for the referee for the careful reading and for the suggestions to improve the readability of the paper.
The first author is supported by an NSERC grant and the second author is supported by a grant from the MacArthur Foundation, by NSF grant DMS-0649473, and by the NSF Waterman award.

\section{Main theorem}

In what follows we are going to use some standard notations and definitions from algebraic geometry. In Section \ref{alg-geom} we provide the basic definitions and tools we will need from algebraic geometry, although this is only the barest of introductions and we refer the reader to standard textbooks like \cite{harris}, \cite{griffiths}, \cite{Mumford}, \cite{hartshorne} (or general reference works such as \cite{gowers}, \cite{springer}) for a detailed treatment.   

Our main result (proven in Section \ref{main-proof}) is as follows.

\begin{theorem}[Main theorem]\label{main}  Let $k, d \geq 0$ be integers such that $d \geq 2k$, and let $\eps > 0$ and $C_0 \geq 1$ be real numbers.  Let $P$ be a finite collection of distinct points in $\R^d$, let $L$ be a finite collection of real algebraic varieties in $\R^d$, and let $\I \subset I(P,L)$ be a set of incidences between $P$ and $L$.  Assume the following ``pseudoline-type'' axioms:
\begin{itemize}
\item[(i)]  For each $\ell \in L$, $\ell$ is a real algebraic variety, which is the restriction to $\R^d$ of a complex algebraic variety $\ell_\C$ of dimension $k$ and degree at most $C_0$.
\item[(ii)]  If $\ell, \ell' \in L$ are distinct, then there are at most $C_0$ points $p$ in $P$ such that $(p,\ell), (p,\ell') \in \I$.
\item[(iii)]  If $p, p' \in P$ are distinct, then there are at most $C_0$ varieties $\ell$ in $L$ such that $(p,\ell), (p',\ell) \in \I$.  (Note that for $C_0=1$, this is equivalent to (ii).)
\item[(iv)]  If $(p,\ell) \in \I$, then $p$ is a smooth (real) point of $\ell$, with a real tangent space.  In other words, for each $(p,\ell) \in \I$, there is a unique tangent space $T_p \ell$ of $\ell$ at $p$, which is a $k$-dimensional real affine space containing $p$.
\item[(v)]  If $\ell,\ell' \in L$ are distinct, and $p \in P$ are such that $(p,\ell), (p,\ell') \in \I$, then the tangent spaces $T_p \ell$ and $T_p \ell'$ are transverse, in the sense that they only intersect at $p$.
\end{itemize}
Then one has
\begin{equation}\label{ipl} |\I| \leq A |P|^{\frac{2}{3}+\eps} |L|^{\frac{2}{3}} + \frac{3}{2} |P| + \frac{3}{2} |L|
\end{equation}
for some constant $A = A_{k,\eps,C_0}$ that depends only on the quantities $k,\eps,C_0$.
\end{theorem}

\begin{remark}\label{com}  The condition $d \geq 2k$ is natural, as we expect the tangent spaces $T_p \ell$, $T_p \ell'$ in Axiom (v) to be $k$-dimensional\footnote{It is possible for these real tangent spaces to have dimension less than $k$, because they are the restriction of the $k$-dimensional \emph{complex} tangent spaces $T_p \ell_\C$, $T_p \ell'_\C$ to $\R^d$.  In applications, the complex tangent spaces will be complexifications of the real tangent spaces, which are then necessarily $k$-dimensional; see Proposition \ref{complexity}.}; if $d < 2k$, such spaces cannot be transverse in $\R^d$.  As we will see shortly, the most interesting applications occur when $d=2k$ and $k \geq 1$, with $C_0$ being an extremely explicit constant such as $1$, $2$, or $4$, the varieties in $L$ being smooth (e.g. lines, planes, or circles), and the incidences $\I$ comprising all of $I(P,L)$; but for inductive reasons it is convenient to consider the more general possibilities for $d$, $k$, $C_0$, $L$, and $\I$ allowed by the above theorem.    The constants $\frac{3}{2}$ could easily be replaced in the argument by any constant greater than $1$.  We choose the constant here to be less than $2$ so that the bound \eqref{ipl} can be used to control $r$-rich points and lines for $r$ as low as $2$ (though in that particular case, the trivial bounds in Lemma \ref{triv-lem} already suffice).
\end{remark}

In Section \ref{special} we will sketch a simplified special case of the above theorem which can be proven using less of the machinery from algebraic geometry; it is conceivable that many of our applications can be handled by this simpler method. On the other hand we expect that there are further applications where the whole generality of our result is needed.

\subsection{Applications}

Suppose we specialize Theorem \ref{szt-thm} to the case when the varieties in $L$ are $k$-dimensional affine subspaces, such that any two of these subspaces meet in at most one point.  Then one easily verifies that Axioms (i)-(v) hold with $C_0=1$ and $\I := I(P,L)$.  We conclude:

\begin{corollary}[Cheap Szemer\'edi-Trotter for $k$-flats]\label{flat}  Let $\eps > 0$, $k \geq 1$, and $d \geq 2k$.  Then there exists a constant $A = A_{\eps,k} > 0$ such that
\begin{equation}\label{claim-flat}
 |I(P,L)| \leq A |P|^{2/3+\eps} |L|^{2/3} + \frac{3}{2}|P| + \frac{3}{2} |L|
\end{equation}
whenever $P$ is a finite set of points in $\R^d$, and $L$ is a finite set of $k$-dimensional affine subspaces in $\R^d$, such that any two distinct spaces in $L$ intersect in at most one point.
\end{corollary}

Except for the $\eps$ loss, this answers a conjecture of T\'oth \cite[Conjecture 3]{To} affirmatively.  The hypothesis $d \geq 2k$ can be dropped for the trivial reason that it is no longer possible for the $k$-dimensional subspaces in $L$ to intersect each other transversely for $d < 2k$, but of course the result is not interesting in this regime.

If $r \geq 2$, and $L$ is a collection of $k$-dimensional affine subspaces, define an \emph{$r$-rich point} to be a point that is incident to at least $r$ subspaces in $L$.  If we apply \eqref{claim-flat} to the set $P$ of $r$-rich points in a standard manner, we have
$$ r |P|\leq |I(P,L)| \leq A |P|^{2/3+\eps} |L|^{2/3} + \frac{3}{2}|P| + \frac{3}{2} |L|.$$
As $r \geq 2>\frac{3}{2}$, we may absorb the $\frac{3}{2} |P|$ term onto the left-hand side, and conclude that if $L$ is any finite collection $k$-dimensional affine subspaces in $\R^d$, any two of which intersect in at most one point, then the number of $r$-rich points is $O_\eps( \frac{|L|^{2+\eps}}{r^3} + \frac{|L|}{r} )$ for any $\eps > 0$ and $r \geq 2$.

As special cases of Corollary \ref{flat}, we almost recover (but for epsilon losses) the classic Szemer\'edi-Trotter theorem (Theorem \ref{szt-thm}), as well as the complex Szemer\'edi-Trotter theorem of T\'oth \cite{To}.  More precisely, from the $k=2$ case of Corollary \ref{flat}, we have

\begin{corollary}[Cheap complex Szemer\'edi-Trotter]\label{complex}  Let $\eps > 0$ and $d \geq 2$.  Then there exists a constant $A = A_\eps > 0$ such that
\begin{equation}\label{claim}
 |I(P,L)| \leq A |P|^{2/3+\eps} |L|^{2/3} + \frac{3}{2} |P| + \frac{3}{2} |L|
\end{equation}
whenever $P$ is a finite set of points in $\C^2$, and $L$ is a finite set of complex lines in $\C^2$.
\end{corollary}

We will sketch a separate proof of this corollary in Section \ref{special}, in order to motivate the more complicated argument needed to establish Theorem \ref{szt-thm} in full generality.

One can also establish the same bound for the quaternions $\HH$.  Define a \emph{quarternionic line} to be any set in $\HH^2$ of the form $\{ (a,b) + t(c,d): t \in \HH \}$ for some $a,b,c,d \in \HH$ with $(c,d) \neq (0,0)$.  Because $\HH$ is a division ring, we see that any two distinct quarternionic lines meet in at most one point.  After identifying $\HH$ with $\R^4$, we conclude:

\begin{corollary}[Cheap quaternionic Szemer\'edi-Trotter]  Let $\eps > 0$.  Then there exists a constant $A = A_\eps > 0$ such that
\begin{equation}\label{claim-quaternion}
 |I(P,L)| \leq A |P|^{2/3+\eps} |L|^{2/3} + \frac{3}{2} |P| + \frac{3}{2} |L|
\end{equation}
whenever $P$ is a finite set of points in $\HH^2$, and $L$ is a finite set of quaternionic lines in $\HH^2$.
\end{corollary}

We can also replace lines with circles.  Define a \emph{complex unit circle} to be a set of the form $\{ (z,w) \in \C^2: (z-z_0)^2+(w-w_0)^2 = 1 \}$ for some $z_0,w_0 \in \C$.  It is easy to verify that complex unit circles are real algebraic varieties in $\C^2 \equiv \R^4$ of (real) dimension $2$ and (real) degree $2 \times 2 = 4$ (because the real and imaginary parts of the defining equation $(z-z_0)^2+(w-w_0)^2=1$ are both quadratic constraints), that two complex unit circles meet in at most two points, and that two points determine at most two complex unit circles.  It is possible for a point to be incident to two distinct  complex unit circles in such a fashion that their tangent spaces (which are complex lines, or real planes) coincide, when the two circles are reflections of each other across their common tangent space; however, if we first pigeonhole the incidences into $O(1)$ classes, based on the orientation of the radial vector $(z-z_0,w-w_0)$ connecting the center $(z_0,w_0)$ of the complex unit circle to the point $(z,w)$, then we can eliminate these unwanted tangencies.  We conclude

\begin{corollary}[Cheap Szemer\'edi-Trotter for complex unit circles]  Let $\eps > 0$.  Then there exists constants $C>0$ and $A = A_\eps$ such that
\begin{equation}\label{claim-circles}
 |I(P,L)| \leq A |P|^{2/3+\eps} |L|^{2/3} + C |P| + C |L|
\end{equation}
for all finite sets of points $P$ and complex unit circles $L$ in $\C^2$.
\end{corollary}

This gives the following application to the complex unit distance problem:

\begin{corollary}[Complex unit distances]  Let $\eps > 0$.  Then there exists a constant $A = A_\eps > 0$ such that
$$ |\{ ((z,w),(z',w')) \in P \times P: (z-z')^2 + (w-w')^2 = 1 \}| \leq A_\eps |P|^{\frac{4}{3}+\eps}$$
for any finite set of points $P$ in $\C^2$.
\end{corollary}

Indeed, this claim follows from applying \eqref{claim-circles} to the family $L$ of complex unit circles with centers in $P$.  We remark that the real-variable analogue of this result was established by Spencer, Szemer\'edi, and Trotter \cite{SST}.

Elekes used the Szemer\'edi-Trotter theorem to give a good bound on the Sum-Product problem of Erd\H os and Szemer\'edi in \cite{E1}. A variant of the original Sum-Product problem was considered by Chang who proved the following in \cite{Ch}. Let us suppose that ${\mathcal{A}}$ is an $n$-element set of $k\times k$ matrices with real coefficients such that $\det(A-B)\neq 0$ for any distinct $A, B$ elements of ${\mathcal{A}}$. Then\footnote{Here we use the usual notation $C+D=\{a+b | a\in C, b\in D \}$ and $CD=\{ab | a\in C , b\in D\}$.}
\[
|{\mathcal{A}}+{\mathcal{A}}|+|{\mathcal{A}}{\mathcal{A}}|\geq g(|{\mathcal{A}}|)|{\mathcal{A}}|
\]
where $g(n)$ goes to infinity as $n$ grows. In \cite{Ta}, the second author showed that $g(n)$ grows polynomially with $n$.

\begin{corollary}[Sum-Product]
Let us suppose that ${\mathcal{A}}$ is an $n$-element set of $k\times k$ matrices with real coefficients such that $\det(A-B)\neq 0$ for any distinct $A, B$ elements of ${\mathcal{A}}$ and $V, W \subset \R^k$ are $n$-element sets of $k$ dimensional vectors. Then for every $\eps > 0$ there exists a constant $c_{k,\eps} > 0$ independent of ${\mathcal A}$ and $n$ such that
\[
|V+W|+|{\mathcal{A}}W|\geq c n^{5/4-\eps}.
\]
\end{corollary}

\begin{proof}
We will apply Corollary \ref{flat}, where the $k$-flats are given by $\{ (\vec x, \vec y) \in \R^k \times \R^k: \vec{y}=A(\vec{x}-\vec{v})\}$ with $\vec{v}\in V$ and $A\in \mathcal{A}$ and the points are the elements of the Cartesian product $\{V+W\}\times\{{\mathcal{A}}W\}$. Any two flats have at most one common point since $\det(A-B)\neq 0$ and each of them has dimension $k$ in the $2k$-dimensional real space. Any point with coordinates $(\vec{w}+\vec{v}, A\vec{w})$ is incident to $\vec{y}=A(\vec{x}-\vec{v})$. So, we have $n^2$ $k$-dimensional flats and $|V+W||{\mathcal{A}}W|$ points where each flat is incident to at least $n$ points. We can apply Corollary \ref{flat} now to prove our bound.
\end{proof}

\begin{remark} The same argument also applies to matrices and point sets with complex coefficients; we omit the details.  In the complex case with $k=1$, we almost obtain the complex version of the sum-product estimate
$$ |A + A| + |A \cdot A| \geq c |A|^{5/4}$$
obtained by Elekes \cite{E1}, but with an epsilon loss in the exponents.
\end{remark}

We show yet another application which is similar to the previous one. It is about $r$-rich affine transformations. Elekes proposed in \cite{E3} a systematic investigation of the following general problem:

{\em " ... Given a group $G$ of transformations of $\R^d$ and a finite pointset $\mathcal{P}\subset \R^d$ we shall be interested in the number of transformations $\varphi\in G$  which map many points of $\mathcal{P}$ to some other points of $\mathcal{P}$ ... "}

Here we consider affine transformations in $\R^d$. We need some notation. An affine transformation is $r$-rich with respect to $\mathcal{P}$ if $|A(\mathcal{P})\cap \mathcal{P}|\geq r.$  Elekes' question is to bound the number of $r$-rich transformations. A finite set of affine transformations, $\mathcal{A}$ is said to be {\em pairwise independent} if $A^{-1}B$ has at most one fixpoint for any $A,B\in \mathcal{A}$.

\begin{corollary}[Affine Transformations]\label{afftran}
Given an $n$-element pointset $\mathcal{P}\subset \R^d$, and let $\eps > 0$ and $r \geq 2$. Any set $X$ of pairwise independent $r$-rich affine transformations has cardinality at most $A n^{4+\varepsilon}/r^3$, where $A = A_{\eps,d} > 0$ depends only on $\eps$ and $d$.
\end{corollary}

\begin{proof} Each affine transformation in $X$ can be written as $\vec{x}\rightarrow A\vec{x}+\vec{v}$, which we can view as a $d$-flat $\{ (\vec x, \vec y) \in \R^d \times \R^d: \vec y = A \vec x + \vec v\}$ in $\R^d \times \R^d$. Each such flat is incident to at least $r$ points of the Cartesian product $\mathcal{P}\times \mathcal{P}$. There are $|X|$ $d$-dimensional flats and $n^2$ points where each flat is incident to at least $r$ points. Any two flats have at most one common point since the transformations are pairwise independent (two or more common points would mean that the corresponding affine transformations are identical on a line). One can apply Corollary \ref{flat} again to prove the bound $|X|\ll n^{4+\varepsilon}/r^3$.
\end{proof}

\subsection{Comparison with existing results}\label{comparison}
The polynomial partitioning method is not the only method to establish incidence bounds between points and varieties.  In particular, there are other methods to obtain cell decompositions which can achieve a similar effect to the decomposition given by the polynomial Ham Sandwich theorem, though the hypotheses on the configuration of points and varieties can be quite different from those considered here.  A model case is when the point set is assumed to be \emph{homogeneous}, which roughly speaking means that the point set resembles a perturbation of a grid.  In such cases one can use the cubes of the grid to form the cells.  For instance, in \cite{sototh} and \cite{sovu} sharp incidence bounds between a homogeneous set of points and $k$-dimensional subspaces were given. In \cite{laba} sharp point-pseudoplane incidence bounds were proved in $\R^3$. Similar bounds on point-surface incidences were proved for the non-homogenous case by  Zahl \cite{zahl}.

Elekes \cite{E3}, Sharir and Welzl  \cite{sharir}, and Guth and Katz \cite{GK-joints} gave bounds on the number of \emph{joints}. We omit the details but we should mention that  the latter paper in particular uses the polynomial space partition method to give a satisfactory bound on the number of joints.

Pach and Sharir \cite{pach, pach2} considered incidences between points and \emph{pseudolines} in the plane - curves which obey axioms similar to Axioms (ii) and (iii) in Theorem \ref{main}, using crossing number inequalities.  Such methods work particularly well in the plane, but are somewhat difficult to extend to higher dimensions; for some partial results in three and higher dimensions, see \cite{aronov}.

In the hyperplane case $k=d-1$, sharp incidence bounds were obtained in \cite{elekes} (of course, the transversality hypothesis needs to be modified in this regime).

In \cite{iosevich}, Fourier-analytic methods were used to obtain incidence bounds of Szemer\'edi-Trotter type. In this method, the manifolds $\ell$ are not required to be algebraic varieties, but they are required to obey certain regularity hypotheses relating to the smoothing properties an associated generalized Radon transform (which usually forces them to be fairly high dimensional.  Also, the point set $P$ is assumed to obey a homogeneity assumption.

In \cite{solymosi} a simple proof was given to a Szemer\'edi-Trotter type bound for incidences between complex points and lines, however the point set is assumed to be of Cartesian product of the form $A \times B \subset \C^d$.

\section{A special case}\label{special}

Before we prove Theorem \ref{main}, we illustrate the key elements of the proof by sketching the proof of the (cheap) complex Szemer\'edi-Trotter Theorem which was stated earlier as Corollary \ref{complex}.  In this low-dimensional setting one can avoid an induction on dimension, instead using the crossing number machinery of Sz\'ekely \cite{szekely} to deal with the contribution of various lower dimensional objects.   The reader who is impatient to get to the proof of the full theorem may skip this section if desired.

\medskip
Let $C_1$ be a large constant to be chosen later, and let $C_3$ be an even larger constant (depending on $C_1$) to be chosen later.  (The subscripts here are chosen to conform to the notation of subsequent sections.) We will show that
\begin{equation}\label{incidence}
|\I| \leq C_3n^{2/3+\varepsilon}m^{2/3} +C_1(n+m)
\end{equation}
for all $n,m \geq 0$ and sets $P$, $L$ of points and complex lines in $\C^2$ with $|P|=n$ and $|L|=m$, by induction on $n$.  The claim is trivial when $n=0$, so suppose that $n \geq 1$, and that the claim has already been proven for smaller $n$.

A well-known Cauchy-Schwarz argument (based on the fact that two lines determine at most one point) gives the bounds
$$ |\I|  \ll n^{1/2} m + n, m^{1/2} n + m$$
(see Lemma \ref{triv-lem} below).  As a consequence, we may restrict attention to the regime
\begin{equation}\label{reg}
m^{1/2} \ll n \ll m^2
\end{equation}
for the purposes of establishing \eqref{incidence}.

One can suppose that every point has at least half of the average incident lines w.l.o.g. (If a point has less than $|\I|/2m$ lines incident to it then simply remove it from $P$.) Then counting the points gives a lower bound on the incidences, e.g. if a region of $\C^2$ contains $k$ points then there are at least $k|\I|/2m$ incidences between the $k$ points and the set $L$.

\medskip

Let us identify $\C$ with $\R^2$ and partition $\C^2=\R^4$ into $M$ cells plus a boundary hypersurface, where $M$ is to be determined later. One can use a degree $D<(12M)^{1/4}$ polynomial $Q$ so that $\R^4\setminus\{Q=0\}$ has $M$ components and no component contains more than $n/M$ points:
$$ \R^{4} = \{ Q = 0 \} \cup \Omega_1 \cup \ldots \cup \Omega_M$$
The existence of such polynomial follows from Corollary \ref{cell-decomp} below.

\medskip

If most of the incidences are in $\R^4\setminus \{Q=0\}$ (inside a cell) then a simple double-counting argument gives the desired bound. This is the case when we are going to use the induction hypothesis. If most of the incidences are on $\{Q=0\}$ then we will bound the number of incidences directly.

Let $\ell \in L$ be one of the complex lines.  Assume first that $\ell$ is not on the surface $\{Q=0\}$.  We will now bound the number of cells that $\ell$ intersects as follows.  We can parameterize $\ell$ in $\C^2$ as
$$ \ell = \{ (z,Az+B): z \in \C \}$$
for some complex numbers $A, B$ (ignoring for the sake of the sketch the ``vertical'' case when $A$ is infinite).  Writing $\C^2$ as $\R^4$, this becomes
$$ \ell = \{ (s,t,as-bt+c,at+bs+d): s,t \in \R \}$$
for some real numbers $a,b,c,d$.  The polynomial $Q(x_1,x_2,x_3,x_4)$ restricted to $\ell$ is then a degree $D$ polynomial of variables $s$ and $t$. As $\{Q=0\}$ is a degree $D$ curve, the number of connected components  of $\ell\setminus \{Q=0\}$ is at most $2(\binom{D-1}{2}+1)$ by the Harnack curve theorem \cite{Ha} (note that each component of $\ell\setminus \{Q=0\}$ will contain a component of $\{Q =\pm\eps\}$ if $\eps$ is small enough); thus each line $\ell$ meets at most $D^2$ cells $\Omega_i$.  

Let $L_i$ denotes the set of lines in $L$ that have non-empty intersection with $\Omega_i$. By \eqref{incidence}, the number of incidences in each cell is bounded by 
  \begin{align*}
|\I \cap I( P \cap \Omega_i, L_i )| &\leq C_3 |P \cap \Omega_i|^{\frac{2}{3}+\eps} |L_i|^{\frac{2}{3}} + C_1(|P \cap \Omega_i| + |L_i|) \\
&\leq C_3 (n/M)^{\frac{2}{3}+\eps} |L_i|^{\frac{2}{3}} + C_1(n/M + |L_i|).
\end{align*}
On the other hand, as each line $\ell$ meets at most $D^2$ cells, we have
$$ \sum_{i=1}^M |L_i| \leq D^2|L|$$
and hence by H\"older's inequality
$$ \sum_{i=1}^M |L_i|^{\frac{2}{3}} \leq \left(D^2 |L|\right)^{\frac{2}{3}}M^{\frac{1}{3}}.$$
Inserting this into our preceding bound and using the fact that $D \leq (12 M)^{1/4}$, we conclude that
$$ \sum_{i=1}^M |\I \cap I(P \cap \Omega_i, L_i)| \leq 12^{\frac{1}{3}}M^{-\eps}C_3 n^{\frac{2}{3}+\eps} m^{\frac{2}{3}} + C_1(n + Mm).$$
Using the hypothesis, we thus conclude (if we choose $M \geq (4\times 12^{\frac{1}{3}})^{\frac{1}{\eps}}$, and then choose $C_3$ large enough depending on $C_1,M$ so that the $C_1(n+Mm)$ term here can be absorbed into the main term using \eqref{reg}) that
$$ \sum_{i=1}^M |\I \cap I(P \cap \Omega_i, L_i)| \leq \frac{1}{2} C_3 n^{\frac{2}{3}+\eps} m^{\frac{2}{3}}.$$

The above estimate handles all the incidences that lie outside of $\{Q=0\}$.
To close the induction, it will thus suffice (if $C_3, C_1$ are chosen large enough) to show that
\begin{equation}\label{illustration}
|\I \cap I(P \cap \{Q=0\}, L)| \ll_{M} n^{\frac{2}{3}+\eps} m^{\frac{2}{3}} + n + m.
\end{equation}
\medskip

To establish this, we first perform a technical decomposition to avoid the issue of singular points on the hypersurface $\{Q=0\}$.  More specifically, we introduce a sequence of hypersurfaces $S_0, S_1, \ldots, S_D$ by setting $S_0:=\{Q=0\}$, $Q_0:=Q$, and 
$$Q_{i+1}:=\sum_{j=1}^4\displaystyle \alpha_j^{(i)}\frac{\partial}{\partial x_j}Q_i$$
for $1 \leq i \leq D$, which defines the surface $S_{i+1}=\{Q_{i+1}=0\}$; here the $\alpha_j^{(i)}$ are generic reals.   All of the $S_i$ have degree at most $D$. For each point $p\in P \cap \{Q=0\}$ there is an index $i=\operatorname{ind}(p)$ between $0$ and $D$ which is the first $i$ when $p\not\in S_{i+1}$. Note that this implies that $p$ is a smooth point of $S_i$, so that the tangent space to $S_i$ at $p$ is three-dimensional, and in particular can contain at most one complex line.  In particular, at most one complex line incident to $p$ is on the hypersurface $S_i$; all other incident complex lines intersect $S_i$ in a one-dimensional curve. Let us consider the set of points $P_i=\{p\in P, \operatorname{ind}(p)=i\}$.  We bound the number of incidences $\I_i$ between $P_i$ and $L$. We will follow Sz\'ekely's method \cite{szekely}. The complex lines on the surface $S_i$ give no more than $|P_i|$ incidences, so it is enough to consider incidences between $P_i$ and complex lines intersecting $S_i$ in a curve. The degree of the intersection curve is at most $D$. Let us project the points $P_i$ and the intersection curves onto a generic plane. We will define a geometric graph $G^{(i)}$ drawn on this generic plane so that the vertex set is the projection of $P_i$.  If a curve is incident to $s+D^2/2$ vertices then we can draw at least $s$ edges along (a component of) the curve without multiple edges. The total number of edge intersections is bounded by B\'ezout's theorem, it is not more than $\binom{m}{2}D^2$. On the other hand one can apply the crossing number inequality \cite{ACNSz,Le} to get a lower bound on the number of crossings. The graph $G^{(i)}$ has at least $|\I_i|-mD^2/2$ edges.  If $|\I_i|-mD^2/2\geq 4|P_i|$ then 

$$\binom{m}{2}D^2\geq \frac{(|\I_i|-mD^2/2)^3}{64|P_i|^2}$$

So either $|\I_i|\leq mD^2/2 + 4|P_i|$ or $|\I_i|\leq 4(mD|P_i|)^{2/3}$. Summing the number of incidences over the $P_i$, $i=1,\ldots,D$ gives the desired inequality (\ref{illustration}) (noting that the implied constants can depend on $M$ and hence on $D$).

\begin{remark} One can also view this argument from a recursive perspective rather than an inductive one.  With this perspective, one starts with a collection of $n$ points and $m$ lines and repeatedly passes to smaller configurations of about $n/M$ points and a smaller number of lines as well.  Thus, one expects about $\log_M n$ iterations in procedure.  Collecting all the bounds together to obtain a final bound of the form $A n^{2/3} m^{2/3}$
(ignoring the lower order terms $n,m$ for now), we see that with each iteration, the constant $A$ increases by a bounded multiplicative factor (independent of $M$), assuming that $A$ was chosen sufficiently large depending on $M$.  Putting together these increases, one obtains a final value of $A$ of the shape $C_M \times C^{\log_M n}$; letting $M$ become large, this gives bounds of the shape $C_\eps n^\eps$ as claimed.  The key point is that the main term in the estimate only grows by a constant factor independent of $M$ with each step of the iteration; the lower order terms, on the other hand, are permitted to grow by constants depending on $C_M$, as they can be absorbed into the main term (using the reduction to the regime \eqref{reg}). 
\end{remark}

\medskip

\section{Some algebraic geometry}\label{alg-geom}

In this section we review some notation and facts from algebraic geometry that we will need here.  Standard references for this material include  \cite{harris}, \cite{griffiths} or \cite{Mumford}.

It will be convenient to define algebraic geometric notions over the field $\C$, as it is algebraically complete. However, for our applications we will only need to deal with the real points of algebraic sets.

\begin{definition}[Algebraic sets]  Let $d \geq 1$ be a dimension.  An \emph{algebraic set} in $\C^d$ is any set of the form
$$ \{ x \in \C^d: P_1(x) =\ldots=P_m(x) = 0\}$$
where $P_1,\ldots,P_m: \C^d \to \C$ are polynomials.  If one can take $m=1$, we call the algebraic set a \emph{hypersurface}.  An algebraic set is \emph{irreducible} if it cannot be expressed as the union of two strictly smaller algebraic sets.  An irreducible algebraic set will be referred to as an \emph{algebraic variety}, or \emph{variety} for short.

The intersection of any subset of $\C^d$ with $\R^d$ will be referred to as the \emph{real points} of that subset.  A \emph{real algebraic variety} is the real points $V_\R$ of a complex algebraic variety $V$.\footnote{Strictly speaking, because two different complex varieties may have the same real points, a real algebraic variety should really be viewed as a pair $(V_\R,V)$ rather than just the set $V_\R$; however, we will abuse notation and identify a real algebraic variety with the set $V_\R$ of real points.}
\end{definition}

If $V$ is a variety in $\C^d$, we can define the \emph{dimension} $\dim(V)=r$ of $V$ to be the largest natural number for which there exists a sequence
$$ \emptyset \neq V_0 \subsetneq V_1 \subsetneq \ldots \subsetneq V_r = V$$
of varieties between $\emptyset$ and $V$.  There are many alternate definitions of this quantity.  For instance, $r$ is also the transcendence degree of the function field $\C(V)$ of $V$ (see \cite[\S I.7]{Mumford}).  Thus, for instance $\dim(\C^d) = d$, and one has $\dim(V) \leq \dim(W)$ whenever $V \subset W$ are varieties, with the inequality being strict if $V \neq W$; in particular, for $V\subset\C^d$, $\dim(V)$ is an integer between $0$ and $d$.  The quantity $d-\dim(V)$ is called the \emph{codimension} of $V$.

It is known (see e.g. \cite[\S 1.3]{griffiths}) that to any $r$-dimensional variety one can associate a unique natural number $D$, called the \emph{degree} of $V$, with the property that almost every codimension $r$ affine subspace of $\C^d$ intersects $V$ in exactly $D$ points.  Thus, for instance, if $P: \C^d \to \C$ is an irreducible polynomial of degree $D$, then the hypersurface $\{P=0\}$ has dimension $d-1$ (and thus codimension $1$) and degree $D$.  

We do not attempt to define the notions of degree and dimension directly for real algebraic varieties, as there are some subtle issues that arise in this setting (see e.g. \cite{roy}).  However, in our applications every real algebraic variety will be associated with a complex one, which of course will carry a notion of degree and dimension.  Later on (by using Proposition \ref{complexity} below) we will see that we may easily reduce to the model case in which the real algebraic varieties have full dimension inside their complex counterparts, in the sense that the real tangent spaces have the same dimension as the complex ones.

Every algebraic set can be uniquely decomposed as the union of finitely many varieties, none of which are contained in any other (see e.g. \cite[Proposition I.5.3]{Mumford}).  We define the dimension of the algebraic set to be the largest dimension of any of its component varieties.

If $V$ is an $r$-dimensional variety in $\C^d$, and $P: \C^d \to V$ is a polynomial which is not identically zero on $V$, then every component of $V \cap \{P=0\}$ has dimension $r-1$ (see \cite[\S I.8]{Mumford}).

A basic fact is that the degree of a variety controls its complexity\footnote{Here, we use the term ``complexity'' informally to refer to the number and degree of polynomials needed to define the variety.}:

\begin{lemma}[Degree controls complexity]\label{dec}  Let $V$ be an algebraic variety in $\C^d$ of degree at most $D$.  Then we can write
$$ V = \{ x \in \C^d: P_1(x) =\ldots=P_m(x) = 0\}$$
for some $m = O_{d,D}(1)$ and some polynomials $P_1,\ldots,P_m$ of degree at most $D$.
\end{lemma}

\begin{proof} See \cite[Theorem A.3]{bgt} or \cite{mum0}.  Indeed, one can take $P_1,\ldots,P_m$ to be a linear basis for the vector space of all the polynomials of degree at most $D$ that vanish identically on $V$.
\end{proof}

We have the following converse:

\begin{lemma}[Complexity controls degree]\label{cont} Let
$$ V = \{ x \in \C^d: P_1(x) =\ldots=P_m(x) = 0\}$$
for some $m \geq 0$ and some polynomials $P_1,\ldots,P_m: \C^d \to \C$ of degree at most $D$.  Then $V$ is the union of $O_{m,D,d}(1)$ varieties of degree $O_{m,D,d}(1)$.
\end{lemma}

\begin{proof} See \cite[Lemma A.4]{bgt} (to obtain the decomposition) and \cite[Lemma 3.5]{bgt} (to bound the degree).
\end{proof}

A \emph{smooth point} of a $k$-dimensional algebraic variety $V$ is an element $p$ of $V$ such that $V$ can be locally described by a smooth $k$-dimensional complex manifold in a neighborhood of $p$.  Points in $V$ that are not smooth will be called \emph{singular}.  We let $V^{\operatorname{smooth}}$ denote the smooth points of $V$, and $V^{\operatorname{sing}} := V \backslash V^{\operatorname{smooth}}$ denote the singular points.

It is well known that ``most'' points in an algebraic variety $V$ are smooth (see e.g. \cite[Theorem 5.6.8]{taylor}).  In fact, we have the following quantitative statement:

\begin{proposition}[Most points smooth]\label{most}  Let $V$ be a $k$-dimensional algebraic variety in $\C^d$ of degree at most $D$.   Then $V^{\operatorname{sing}}$ can be covered by $O_{D,d}(1)$ algebraic varieties in $V$ of dimension at most $k-1$ and degree $O_{D,d}(1)$.
\end{proposition}

\begin{proof}   We will induct on the codimension $d-k$ of $V$.  The codimension zero case $k=d$ is trivial.
Now suppose that $d-k \geq 1$, and the claim has already been proven for smaller values of the codimension.  As $V$ is a $k$-dimensional variety of degree at most $D$, almost every projection of $V$ from $\C^d$ to $\C^{d-1}$ will be dense in a $k$-dimensional variety of degree at most $D$ (the density can be seen by looking at the neighbourhood of a smooth point of $V$).  Applying a linear transformation, we may thus assume without loss of generality that the projection of $V$ under the standard projection $\pi: \C^d \to \C^{d-1}$ defined by $\pi(x_1,\dots,x_d) := (x_1,\dots,x_{d-1})$ is a dense subset of a $k$-dimensional variety $V'$ of degree at most $D$.  By induction hypothesis, $V'$ is smooth outside of $O_{D,1}(1)$ varieties in $V'$ of dimension at most $k-1$ and degree $O_{D,d}(1)$.  

By Lemma \ref{dec}, $V$ is the zero locus of some polynomials $P_1,\dots,P_m$ of degree at most $D$ with $m = O_{D,d}(1)$.  Since $V$ has smaller dimension than $\pi^{-1}(V')$, at least one of the polynomials does not vanish identically on $\pi^{-1}(V')$.  Without loss of generality we may assume that $P_1$ does not vanish identically on $\pi^{-1}(V')$.  If $\partial_{x_d} P_1$ vanished identically on $\pi^{-1}(V')$, then as $P_1$ vanishes on $V$, we see the fundamental theorem of calculus $P_1$ would vanish on $\pi^{-1}(\pi(V))$.  But this is a dense subset of $\pi^{-1}(V')$ and so $P_1$ vanishes on $\pi^{-1}(V')$, a contradiction.  Thus $\partial_{x_d} P_1$ does not vanish identically on $\pi^{-1}(V')$.  If $\partial_{x_d} P_1$ vanishes identically on $V$, we differentiate again and repeat the above argument.  Since $\partial_{x_d}^j P_1$ vanishes identically on $\pi^{-1}(V')$ for $j > D$, we conclude that there exists $0 \leq j \leq D$ such that $\partial_{x_d}^j P_1$ vanishes identically on $V$, but such that $\partial_{x_d}^j P_1$ does not vanish identically on $\pi^{-1}(V')$, and $\partial_{x_d}^{j+1} P_1$ does not vanish identically on $V$.  In the neighbourhood of any point $(x_1,\dots,x_d)$ in $V$ with $(x_1,\dots,x_{d-1}) \in (V')^{\operatorname{smooth}}$ and $\partial_{x_d}^{j+1} P_1(x_1,\dots,x_d) \neq 0$, we see by applying the implicit function theorem to $\partial_{x_d}^j P_1$ that $(x_1,\dots,x_d)$ is smooth.  We conclude (using Lemma \ref{dec} and Lemma \ref{cont}) that $V$ is smooth outside of $O_{D,d}(1)$ varieties of dimension at most $k-1$ and degree $O_{D,d}(1)$, closing the induction.
\end{proof}

We may iterate this proposition (performing an induction on the dimension $k$) to obtain

\begin{corollary}[Decomposition into smooth points]\label{decomp} Let $V$ be a $k$-dimensional algebraic variety in $\C^d$ of degree at most $D$.   Then one can cover $V$ by $V^{\operatorname{smooth}}$ and $O_{D,d}(1)$ sets of the form $W^{\operatorname{smooth}}$, where each $W$ is an algebraic variety in $V$ of dimension at most $k-1$ and degree $O_{D,d}(1)$.
\end{corollary}

Finally, we address a technical point regarding the distinction between real and complex algebraic varieties.  If $p \in \ell \subset \R^d$ is a smooth real point of a $k$-dimensional complex algebraic variety $\ell_\C \in \C^d$, then it must have a $k$-dimensional complex tangent space $T_p \ell_\C \subset \C^d$, by definition.  However, its real tangent space $T_p \ell := T_p \ell_\C \cap \R^d$ may have dimension smaller than $k$.  For instance, in the case $k=1, d=2$, the complex line $\ell_\C := \{ (z,w) \in \C^2: z = iw \}$ has a smooth point at $(0,0)$ with a one-dimensional complex tangent space (also equal to $\ell_\C$), but the real tangent space is only zero-dimensional.  Of course, in this case, the real portion $\ell := \ell_\C \cap \R^2$ of the complex line is just a zero-dimensional point.  This phenomenon generalizes:

\begin{proposition}\label{complexity}  Let $V$ be a $k$-dimensional algebraic variety in $\C^d$ of degree at most $D$.  Then at least one of the following statements is true:
\begin{itemize}
\item[(i)]  The real points $V_\R$ of $V$ are covered by the smooth points $W^{\operatorname{smooth}}$ of $O_{D,d}(1)$ algebraic varieties $W$ of dimension at most $k-1$ and degree $O_{D,d}(1)$ which are contained in $V$.
\item[(ii)]  For every smooth real point $p \in V^{\operatorname{smooth}}_\R$ of $V$, the real tangent space $T_p V_\R := T_p V \cap \R^d$ is $k$-dimensional (and thus $T_p V$ is the complexification of $T_p V_\R$).  In particular, $V^{\operatorname{smooth}}_\R$ is $k$-dimensional.
\end{itemize}
\end{proposition}

\begin{proof}  By Lemma \ref{dec}, $V$ can be cut out by polynomials $P_1,\ldots,P_m$ of degree $O_{D,d}(1)$ with $m = O_{D,d}(1)$.  Let $\tilde V$ be the algebraic set cut out by both $P_1,\ldots,P_m$ and their complex conjugates $\overline{P_1},\ldots,\overline{P_m}$, defined as the complex polynomials whose coefficients are the conjugates of those for $P_1,\ldots,P_m$.  Clearly $V$ and $\tilde V$ have the same real points.  If $\tilde V$ is strictly smaller than $V$, then it has dimension at most $k-1$, and by Lemma \ref{cont} and Corollary \ref{decomp} we are thus in case (i).  Now suppose instead that $V$ is equal to $\tilde V$.  Then at every smooth point $p$ of $V$ (or $\tilde V$), the $k$-dimensional complex tangent space $T_p V = T_p \tilde V$ is cut out by the orthogonal complements of the gradients $\nabla P_1,\ldots,\nabla P_m$ and their complex conjugates.  As such, it is manifestly closed with respect to complex conjugation, and is thus the complexification of its real counterpart $T_p \tilde V_\R$.  We are thus in case (ii).
\end{proof}

Note that by iterating the above proposition, we may assume that each of the varieties $W$ occuring in case (i) obey the properties stated in case (ii).

\section{Proof of main theorem}\label{main-proof}

We now prove Theorem \ref{main}.  We will prove this theorem by an induction on the quantity $d+k$.  The case $d+k=0$ is trivial, so we assume inductively that $d+k \geq 1$ and that the claim has already been proven for all smaller values of $d+k$.

We can dispose of the easy case $k=0$, because in this case $L$ consists entirely of points, and in particular each $\ell$ in $L$ is incident to at most one point in $P$, giving the bound $|\I| \leq |L|$ which is acceptable.  Hence we may assume that $k \geq 1$, and thus $d \geq 2$.

We now perform a technical reduction to eliminate some distinctions between the real and complex forms of the variety $\ell$.  We may apply Proposition \ref{complexity} to each variety $\ell$.  Those varieties $\ell$ which obey conclusion (i) of that proposition can be easily handled by the induction hypothesis, since the varieties $W$ arising from that conclusion have dimension strictly less than $k$.  Thus we may restrict attention to varieties which obey conclusion (ii), namely that for every smooth real point $p$ in $\ell$ (and in particular, for all $(p,\ell) \in \I$), the real tangent space $T_p \ell$ has full dimension $k$.

We now divide into two subcases: the case $d>2k$ and the case $d=2k$, and deduce each case from the induction hypothesis.  The main case is the latter; the former case will be obtainable from the induction hypothesis by a standard projection argument.

\subsection{The case of excessively large ambient dimension}\label{large-sec}

We first deal with the case $d>2k$.  Here, we can use a generic projection argument to deduce the theorem from the induction hypothesis.  Indeed, fix $\eps, C_0, P, L$, and let $\pi: \R^d \to \R^{2k}$ be a generic linear transformation (avoiding a finite number of positive codimension subvarieties of the space $\operatorname{Hom}(\R^d,\R^{2k}) \equiv \R^{2kd}$ of linear transformations from $\R^d$ to $\R^{2k}$).  Since $k \geq 1$, we see that for any two distinct points $p, p' \in P$, the set of transformations $\pi$ for which $\pi(p)=\pi(p')$ has positive codimension in $\operatorname{Hom}(\R^d,\R^{2k})$.  Hence for generic $\pi$, the map from $P$ to $\pi(P)$ is bijective.

Similarly, let $\ell, \ell'$ be two distinct $k$-dimensional varieties in $L$.  Then generically $\pi(\ell), \pi(\ell')$ will be $k$-dimensional varieties in $\R^{2k}$.
We claim that for generic $\pi$, the varieties $\pi(\ell), \pi(\ell')$ are distinct.  Indeed, since $\ell, \ell'$ are distinct varieties of the same dimension, there exists a point $p$ on $\ell'$ that does not lie on $\ell$.  Since $\ell$ has codimension $d-k > d-2k$, we conclude that for generic $\pi$, the $d-2k$-dimensional plane $\pi^{-1}(\pi(p))$ does not intersect $\ell'$, and the claim follows.  Thus the map $\pi: L \to \pi(L)$ is bijective.

Next, let $(p,\ell) \in \I$. Clearly, $\pi(p) \in \pi(\ell)$, and so $\pi$ induces a bijection from $\I$ to some set of incidences $\pi(\I) \subset I(\pi(P), \pi(L))$.

To sumarise, we have established that for generic $\pi$, 
\begin{align*}
|\pi(P)| &= |P| \\
|\pi(L)| &= |L| \\
|\pi(\I)| &= |\I|.
\end{align*}
We would now like to apply the induction hypothesis to the configuration of points $\pi(P)$, varieties $\pi(L)$, and incidences $\pi(\I)$, which will clearly establish the desired claim.  To do this, we need to verify that $\pi(P)$, $\pi(L)$, $\pi(\I)$ obey the axioms (i)-(v) for generic $\pi$.

If $\ell \in L$, then $\ell$ is a $k$-dimensional variety of degree at most $C_0$, and so generically $\pi(\ell)$ will also be a $k$-dimensional variety of degree at most $C_0$, which gives Axiom (i).

Axioms (ii) and (iii) for $\pi(P), \pi(L), \pi(\I)$ are generically inherited from those of $P, L, \I$ thanks to the bijection between $\I$ and $\pi(\I)$.

Now we turn to Axiom (iv).  Fix $(p,\ell) \in \I$, then $p$ is a smooth point of $\ell$.  Since $d > 2k$ and $k \geq 1$, it will generically hold that the codimension $2k$ affine space $\pi^{-1}(\pi(p))$ will intersect the codimension $d-k$ variety $\ell$ only at $p$.  Thus, $\pi(p)$ will generically be a smooth point of $\pi(\ell)$, which gives Axiom (iv).

Finally, if $(p,\ell), (p,\ell') \in \I$, then by hypothesis, the $k$-dimensional tangent spaces $T_p \ell, T_p \ell'$  are transverse and thus span a $2k$-dimensional affine space through $p$.  Generically, $\pi$ will be bijective from this space to $\C^{2k}$, and thus $T_{\pi(p)} \pi(\ell) = \pi(T_p \ell)$ and $T_{\pi(p)} \pi'(\ell') = \pi(T_p \ell')$ remain transverse.  This gives Axiom (v).

Now that all the axioms are verified, the induction hypothesis gives
$$ |\pi(\I)| \leq A |\pi(P)|^{\frac{2}{3}+\eps} |\pi(L)|^{\frac{2}{3}} + \frac{3}{2} |\pi(P)| + \frac{3}{2} |\pi(L)|$$
for generic $\pi$ and some $A$ depending only on $k,\eps,C_0$, and the claim follows.

\subsection{The case of sharp ambient dimension}

It remains to handle the case when $d=2k$ and $k \geq 1$.  Here we fix $d, k, C_0, \eps$, and allow all implied constants in the asymptotic notation to depend on these parameters.  We will need some additional constants
$$ C_3 > C_2 > C_1 > C_0,$$
where
\begin{itemize}
\item $C_1$ is assumed to be sufficiently large depending on $C_0$, $k$ and $\eps$;
\item $C_2$ is assumed to be sufficiently large depending on $C_1, C_0, k$, and $\eps$; and
\item $C_3$ is assumed to be sufficiently large depending on $C_2,C_1,C_0, k$, and $\eps$.
\end{itemize}

For suitable choices of $C_1,C_2,C_3$, we will establish the inequality
\begin{equation}\label{scl}
 |\I| \leq C_3 |P|^{2/3+\eps} |L|^{2/3} + \frac{3}{2} |P| + \frac{3}{2} |L|
\end{equation}
by an induction on the number of points $P$.  The inequality is trivial for $|P|=0$, so we assume that $|P| \geq 1$ and that the claim has been proven for all smaller sets of points $P$.

Suppose $P, L, \I$ obey all the specified axioms.  We begin with two standard trivial bounds:

\begin{lemma}[Trivial bounds]\label{triv-lem}   We have
\begin{equation}\label{itriv}
|\I| \leq C_0^{1/2} |P| |L|^{1/2} + |L|
\end{equation}
and
$$ |\I| \leq C_0^{1/2} |L| |P|^{1/2} + |P|.$$
\end{lemma}

\begin{proof}  Consider the set $\Sigma$ of triples $(p,p',\ell) \in P \times P \times L$ such that $(p,\ell), (p',\ell) \in \I$.  If $p \neq p'$, then from Axiom (iii) this pair contributes at most $C_0$ triples to $\Sigma$, while the case $p=p'$ contributes exactly $|\I|$ triples.  We thus have
$$ |\Sigma| \leq C_0 |P|^2 + |\I|.$$
 On the other hand, from Cauchy-Schwarz 
$$ |\Sigma| \geq \frac{|\I|^2}{|L|}$$
and thus
$$ |\I|^2 - |L| |\I| \leq C_0 |P|^2 |L|.$$
We rearrange this as
$$ \left(|\I|-\frac{|L|}{2}\right)^2 \leq \frac{|L|^2}{4} + C_0 |P|^2 |L|.$$
The right-hand side is bounded by $\left(\frac{|L|}{2} + C_0^{1/2} |P| |L|^{1/2}\right)^2$, which gives \eqref{itriv}.  The second bound is proven similarly by swapping the roles of $P$ and $L$ (and using Axiom (ii) instead of Axiom (iii)).
\end{proof}

One can also view this lemma as a special case of the classical K\~ov\'ar-S\'os-Tur\'an theorem \cite{kst}.

In view of this lemma, we can obtain the desired bound \eqref{scl} whenever we are outside the regime
\begin{equation}\label{regime}
 C_2 |P|^{1/2} \leq |L| \leq C_2^{-1} |P|^2
\end{equation}
provided that $C_2$ is large enough depending on $C_0$, and $C_3$ is large enough depending on $C_0, C_2$.  Thus we may assume without loss of generality that we are in the regime \eqref{regime}.

Following Guth and Katz \cite{GK}, the next step is to apply the polynomial ham sandwich theorem of Stone and Tukey \cite{ST}. We will sketch their argument here. For more details we refer to the excellent review of the polynomial decomposition method in \cite{kaplan}.

\begin{theorem}[Polynomial ham sandwich theorem]\label{hst} Let $D, M, d \geq 1$ be integers with $M = \binom{D+d}{d} -1$, and
let $S_1, \ldots, S_M$ be finite sets of points of $\R^d$.  Then there exists a polynomial $Q: \R^d \to \R$ of degree at most $D$ such that the real algebraic hypersurface $\{ x \in \R^d: Q(x)=0\}$ bisects each of the $S_1,\ldots,S_M$, in the sense that
$$ |\{ x \in S_i: Q(x) < 0 \}| \leq \frac{1}{2} |S_i|$$
and
$$ |\{ x \in S_i: Q(x) > 0 \}| \leq \frac{1}{2} |S_i|$$
for all $i=1,\ldots,M$.
\end{theorem}

As observed in \cite{GK}, we may iterate this to obtain

\begin{corollary}[Cell decomposition]\label{cell-decomp}\cite{GK}  Let $D, d \geq 1$ be integers, and let $P$ be a finite set of points in $\R^d$.  Then there exists a decomposition
$$ \R^d = \{ Q = 0 \} \cup \Omega_1 \cup \ldots \cup \Omega_M$$
where $Q: \R^d \to \R$ is a polynomial of degree at most $D$, $M = O_d(D^d)$, and $\Omega_1,\ldots,\Omega_M$ are open sets bounded by $\{Q=0\}$ (i.e. the topological boundary of the $\Omega_i$-s are contained in $\{Q=0\}$), such that $|P \cap \Omega_i| = O_d(|P|/D^d)$ for each $1 \leq i \leq M$.
\end{corollary}

\begin{proof}  We may assume that $D$ is large depending on $d$, as the claim is trivial otherwise.
By Theorem \ref{hst} and induction on $M$, there is a constant $A = A_d > 0$ such that for every power of two $M \geq 1$ we can find a partition
$$ \R^d = \{Q_M = 0 \} \cup \Omega_{M,1} \cup \ldots \cup \Omega_{M,M}$$
where $Q_M: \R^d \to \R$ is a polynomial of degree at most $A M^{1/d}$, and $\Omega_{M,1},\ldots,\Omega_{M,M}$ are open sets bounded by $\{Q_M=0\}$ such that $|P \cap \Omega_i| \leq |P|/M$ for each $1 \leq i \leq M$.  The claim follows by selecting $M$ comparable to a large multiple of $D^{1/d}$.
\end{proof}

We apply this corollary to our situation with $D := C_1$, to give a decomposition
$$ \R^{2k} = \{ Q = 0 \} \cup \Omega_1 \cup \ldots \cup \Omega_M$$
where $Q: \R^{2k} \to \R$ has degree at most $C_1$, $M = O( C_1^{2k} )$, and the $\Omega_i$ are open sets bounded by $\{Q=0\}$ such that $|P \cap \Omega_i| = O( |P| / C_1^{2k} )$ for each $i$.  (Recall that we allow implied constants in the $O()$ notation to depend on $k$.) In particular, we have $|P \cap \Omega_i| < |P|$ for each $i$, which will allow us to apply the induction hypothesis to each $P \cap Q_i$.

For each $1 \leq i \leq M$, let $L_i$ be the sets $\ell$ in $L$ that have non-empty intersection with $\Omega_i$.  Clearly 
$$ |\I| = |\I \cap I(P \cap \{Q=0\}, L)| + \sum_{i=1}^M |\I \cap I( P \cap \Omega_i, L_i )|.$$
We first estimate the latter sum.  Applying the induction hypothesis, we have
\begin{align*}
|\I \cap I( P \cap \Omega_i, L_i )| &\leq C_3 |P \cap \Omega_i|^{\frac{2}{3}+\eps} |L_i|^{\frac{2}{3}} + \frac{3}{2}|P \cap \Omega_i| + \frac{3}{2}|L_i| \\
&\leq C_1^{-\frac{4k}{3}-2k\eps} C_3 |P|^{\frac{2}{3}+\eps} |L_i|^{\frac{2}{3}} + \frac{3}{2}|P/M| + \frac{3}{2}|L_i|.
\end{align*}
Note the factor of $C_1^{-2k\eps}$ in the main term, which will be crucial in closing the induction.

Let $\ell$ be a variety in $L$, then $\ell$ is $k$-dimensional and has degree $O(1)$.  Meanwhile, the set $\{Q=0\}$ is a hypersurface of degree at most $C_1$.  We conclude that $\ell$ either lies in $\{Q=0\}$, or intersects $\{Q=0\}$ in an algebraic set of dimension at most $k-1$.  In the former case, $\ell$ cannot belong to any of the $L_i$.  In the latter case we apply a generalization of a classical result established independently by Oleinik and Petrovsky \cite{OP}, Milnor \cite{Mi}, and Thom \cite{Th}, such that the number of connected components of $\ell \backslash \{Q=0\}$ is at most $O( C_1^k )$; we give a proof of this fact in Theorem \ref{pad}. Recently a more general bound was proved by Barone and Basu \cite{basu}. 

From Theorem \ref{pad} we will use here that $\ell$ can belong to at most $O(C_1^k)$ of the sets $L_i$.  

This implies that
$$ \sum_{i=1}^M |L_i| \ll C_1^k |L|$$
and thus by H\"older's inequality and the bound $M = O(C_1^{2k})$
$$ \sum_{i=1}^M |L_i|^{\frac{2}{3}} \ll C_1^{\frac{4k}{3}} |L|^{\frac{2}{3}}.$$
Inserting this into our preceding bounds, we conclude that
$$ \sum_{i=1}^M |\I \cap I(P \cap \Omega_i, L_i)| \ll C_1^{-2k\eps} C_3 |P|^{\frac{2}{3}+\eps} |L|^{\frac{2}{3}} + |P| + C_1^k |L|.$$
Using the hypothesis \eqref{regime}, we thus conclude (if $C_1$ is large enough depending on $k, \eps$, and $C_2$ is large enough depending on $C_1$, $k$, $\eps$) we have
$$ \sum_{i=1}^M |\I \cap I(P \cap \Omega_i, L_i)| \leq \frac{1}{2} C_3 |P|^{\frac{2}{3}+\eps} |L|^{\frac{2}{3}}.$$
To close the induction, it will thus suffice (again by using \eqref{regime}) to show that
\begin{equation}\label{close}
|\I \cap I(P \cap \{Q=0\}, L)| \ll_{C_1} |P|^{\frac{2}{3}+\eps} |L|^{\frac{2}{3}} + |P| + |L|.
\end{equation}

\begin{remark}  A modification of the above argument shows that if one applied Corollary \ref{cell-decomp} not with a bounded degree $D=C_1$, but instead with a degree $D$ comparable to $(|P|^{2/3} |L|^{-1/3})^{1/k}$, then one could control the contribution of the cell interiors $\Omega_i$ by the trivial bound \eqref{itriv}, rather than the inductive hypothesis, thus removing the need to concede an epsilon in the exponents.  However, the price one pays for this is that the hypersurface $\{Q=0\}$ acquires a much higher degree, and the simple arguments given below to handle the incidences on this hypersurface are insufficient to give good bounds (except in the original Szemer\'edi-Trotter context when $k=1$ and $d=2$).  Nevertheless, it may well be that a more careful analysis, using efficient quantitative bounds on the geometry of high degree varieties, may be able to recover good bounds for this strategy, and in particular in removing the epsilon loss in Theorem \ref{main}.
\end{remark}

For inductive reasons, it will be convenient to prove the following generalisation:

\begin{proposition}  Let the notation and hypotheses be as above.  (In particular, we are assuming Theorem \ref{main} to already be proven for all smaller values of $d+k$, and continue to allow all implied constants to depend on $k$.)  Let $0 \leq r < 2k$, and let $\Sigma$ be an $r$-dimensional variety in $\R^{2k}$ of degree at most $D$.  Then
$$|\I \cap I(P \cap \Sigma, L)| \ll_{D} |P|^{\frac{2}{3}+\eps} |L|^{\frac{2}{3}} + |P| + |L|.$$
\end{proposition}

Clearly, \eqref{close} follows by specialising to the case $r=2k-1$, $D=C_1$, and $\Sigma=\{Q=0\}$.

\begin{proof}  We induct on $r$.  If $r=0$, then $\Sigma$ is a single point, and so each set $\ell$ in $L$ has at most one incidence in $P \cap \Sigma$, giving the net bound $|\I| \leq |L|$, which is acceptable.

Now suppose that $1 \leq r < 2k-1$, and that the claim has already been proven for smaller values of $r$.

We may of course delete all points of $P$ outside of $\Sigma$.  If $p$ is a point in $\Sigma$, then $p$ is either a smooth point in $\Sigma$ or a singular point.  Let us first deal with the contribution of the latter case.  As $\Sigma$ is an $r$-dimensional variety of degree at most $D$, the singular points in $\Sigma$ lie in a union of $O_{D}(1)$ varieties $\Sigma_1,\ldots,\Sigma_m$ of degree $O_D(1)$ and 	dimension strictly less than $r$ (see Proposition \ref{most}).  By the induction hypothesis, we have
$$
|\I \cap I(P \cap \Sigma_i, L)| \ll_{D} |P|^{\frac{2}{3}+\eps} |L|^{\frac{2}{3}} + |P| + |L|$$
for each such variety $\Sigma_i$, and so on summing in $i$ we see that the contribution of the singular points is acceptable.

By deleting the singular points of $\Sigma$ from $P$, we may thus assume without loss of generality that all the points in $P$ are smooth points of $\Sigma$.  For each $\ell \in L$, consider the intersection $\ell \cap \Sigma$.  As $\ell$ is a variety of degree $k$, we see that $\ell$ is either contained in $\Sigma$, or $\ell \cap \Sigma$ will be a algebraic set of dimension strictly less than $k$.

Consider the contribution of the first case when $\ell$ is contained in $\Sigma$.  If we have two distinct incidences $(p,\ell), (p,\ell') \in \I$ such that $\ell, \ell'$ both lie in $\Sigma$, then the $k$-dimensional tangent spaces $T_p \ell, T_p \ell'$ lie in the $r$-dimensional space $T_p \Sigma$.  On the other hand, by Axiom (iv), $T_p \ell$ and $T_p \ell'$ are transverse.  Since $r<2k$, this is a contradiction.  Thus, each point in $P$ is incident to at most one variety $\ell \in L$ that lies in $\Sigma$, and so there are at most $|L|$ incidences that come from this case.

We may thus assume that each variety $\ell \in L$ intersects $\Sigma$ in an algebraic set of dimension strictly less than $k$.  Since $\ell$ has degree at most $C_0$, and $\Sigma$ has degree at most $D$, we see from Corollary \ref{decomp} that $\ell \cap \Sigma$ is the union of the smooth points of $O_{D}(1)$ algebraic varieties of dimension between $0$ and $k-1$ and degree $O_{D}(1)$.

Thus we may write $\ell \cap \Sigma = \bigcup_{k' = 0}^{k-1} \bigcup_{j=1}^J \ell_{k',j}^{\operatorname{smooth}}$ for some $J=O_{D}(1)$, where for each $0 \leq k' < k-1$ and $1 \leq j \leq J$, $\ell_{k',j}$ is either empty, or is an algebraic variety of dimension $k'$ and degree $O_{D}(1)$, and $\ell_{k',j}^{\operatorname{smooth}}$ are the smooth points of $\ell_{k',j}$.  Note that by padding the decomposition with empty varieties, we may assume that $J$ is independent of $\ell$. We may then estimate
$$ |\I \cap I(P \cap \Sigma, L)| \leq \sum_{k'=0}^{k-1} \sum_{j=1}^J |\I_{k',j}|$$
where $\I_{k',j}$ is the set of all incidences $(p,\ell) \in \I$ such that $p \in \ell_{k',j}^{\operatorname{smooth}}$.  It thus suffices to show that
\begin{equation}\label{pell}
 |\I_{k',j}| \ll_{D} |P|^{\frac{2}{3}+\eps} |L|^{\frac{2}{3}} + |P| + |L|
 \end{equation}
for each $k'$ and $j$.

Fix $k'$ and $j$.  Those varieties $\ell \in L$ for which $|\I \cap (P \times \{\ell\})| \leq C_0$ will contribute at most $C_0|L|$ incidences to \eqref{pell}, so we may assume that $|\I \cap (P \times \{\ell\})| > C_0$ for all $\ell \in L$.  By Axiom (ii), this forces the $\ell_{k',j}$ to be distinct.  If we then let $L' = L'_{k',j}$ be the set of all the $\ell_{k',j}$, we can thus identify $\I_{k',j}$ with a subset $\I'$ of $I(P,L')$.
However, by induction hypothesis, Theorem \ref{main} is already known to hold if $k$ is replaced by $k'$ (keeping $d=2k$ fixed).  So, to conclude the argument, it suffices to show that $P, L', \I'$ obey the axioms of Theorem \ref{main}, with $C_0$ replaced by $O_{C_0,D}(1)$. But
Axiom (i) is clear from construction, while Axioms (ii), (iii), (iv), and (v) are inherited from the corresponding axioms for $P, L, \I$.  This closes the induction and proves the lemma.
\end{proof}

The proof of Theorem \ref{main} is now complete.

\begin{remark}  It is reasonable to conjecture that one can set $\eps=0$ in Theorem \ref{main}.  From $k=1$, this can be established using the crossing number inequality \cite{ACNSz, Le} and the Harnack curve theorem \cite{Ha}, following the arguments of Sz\'ekely \cite{szekely}; it is also possible to establish this inequality via the polynomial partitioning method.  For $k=2$, a more careful analysis of the above arguments (using the refinement to the $k=1$ case mentioned above) eventually shows that one can take $A$ to be of the shape $\exp( O_{C_0}(1/\eps) )$; optimizing in $\eps$, one can thus replace the $A |P|^{\eps}$ factor by $\exp( O_{C_0}( \sqrt{\log |P|} ) )$.  However, for $k>2$, the highly inductive nature of the argument causes $A$ to depend on $\eps$ in an iterated exponential manner.
\end{remark}

\begin{remark}
One can construct many examples in which $|I(P,L)|$ is comparable to $|P|^{\frac{2}{3}} |L|^{\frac{2}{3}} + |P| + |L|$ by taking $k$ of the standard point-line configurations in $\R^2$ that demonstrate that the original Szemer\'edi-Trotter theorem (Theorem \ref{szt-thm}) is sharp, and then taking Cartesian products (and increasing the ambient dimension if desired).  It is natural to conjecture that the $\eps$ loss in \eqref{ipl} can be eliminated, but our methods do not seem to easily give this improvement.
\end{remark}

It is possible to drop Axiom (iv), at the cost of making Axiom (v) more complicated.  For any point $p$ on a real algebraic variety $\ell \subset \R^d$, define the \emph{tangent cone} $C_p \ell$ to be the set of all elements in $\R^d$ of the form $\gamma'(0)$, where $\gamma: [-1,1] \to \ell$ is a smooth map with $\gamma(0) = p$.  Note that at a smooth point $p$ of $\ell$, the tangent cone is nothing more than the tangent space $T_p \ell$ (translated to the origin).  However, the tangent cone continues to be well-defined at singular points, while the tangent space is not.

\begin{corollary} The conclusions of Theorem \ref{main} continue to hold if Axiom (iv) is dropped, but the tangent space $T_p \ell$ in Axiom (v)
is replaced by the tangent cone, but where the constant $A$ is now allowed to depend on the ambient dimension $d$ in addition to $k$ and $\eps$.
\end{corollary}

\begin{proof}  (Sketch) We perform strong induction on $k$, assuming that the claim has already been proven for smaller $k$.
For those incidences $(p,\ell)$ in $\I$ for which $p$ is a smooth point of $\ell$, one can apply Theorem \ref{main} to get a good bound, so we may restrict attention to those incidences in which $p$ is a singular point of $\ell$.  But then we can use Corollary \ref{decomp} to cover the singular portion of $\ell$ by $O_{C_0,k,d}(1)$ irreducible components in $\ell$ of dimension at most $k-1$ and degree $O_{C_0,k,d}(1)$.  Applying the induction hypothesis to each of these components (noting that Axioms (i)-(iii) and the modified Axiom (v) are inherited by these components, increasing $C_0$ if necessary) we obtain the claim.
\end{proof}

\appendix

\section{Connected components of real semi-algebraic sets}

In this appendix we give some standard bounds on the number of connected components of various real semi-algebraic sets.  We first bound the number of connected components cut out by a hypersurface $\{ P=0\}$ of bounded degree. This result was proved by independently by Oleinik and Petrovsky \cite{OP}, Milnor \cite{Mi}, and Thom \cite{Th}. (They proved the stronger result that under the conditions of the theorem below the sum of the Betti numbers of $X$ is $O_d(D^d)$.) Recently a more general bound was presented by Barone and Basu \cite{basu}. They gave tight estimates on the dependence on the various parameters which are hidden in the big-Oh notation in Theorem \ref{pad}.

\begin{theorem}\label{conn}  Let $P: \R^d \to \R$ be a polynomial of degree at most $D$ for some $D \geq 1$ and $d \geq 0$.  Then the set $\{ x \in \R^d: P(x) \neq 0 \}$ has $O_d(D^d)$ connected components.
\end{theorem}

\begin{proof}  We proceed by induction on $d$.  The case $d=0$ is trivial, so suppose that $d \geq 1$ and that the claim is already proven for $d-1$.  By a limiting argument, it suffices to show that for any cube $Q$ in $\R^d$, the number of components of $\{ x \in Q: P(x) \neq 0 \}$ has $O_d(D^d)$ connected components, uniformly in $Q$.  The number of components that intersect the boundary $\partial Q$ of $Q$ is $O_d(D^{d-1})$ by the induction hypothesis, so it suffices to control the number of components that lie completely in the interior of $Q$.

Let $C_1,\ldots,C_m$ be the components of $\{ x \in Q: P(x) \neq 0 \}$ that avoid $\partial Q$; our task is to show that $m = O_d(D^d)$.  Observe that each $C_i$ is an open set, with $P$ non-zero in $C_i$ and vanishing on the boundary of $C_i$.  By the continuity of $P$, we can thus find a compact subset $K_i$ of $C_i$ such that on $K_i$, at least one of the maximum or minimum of $P$ is attained only in the interior of $K_i$ (i.e. one either has $\sup_{x \in K_i} P(x) > \sup_{x \in \partial K_i} P(x)$ or $\inf_{x \in K_i} P(x) < \inf_{x \in \partial K_i} P(x)$). In particular, $P$ has at least one critical point in the interior of $K_i$, and so $\nabla P$ has at least one zero in the interior of $K_i$, and so the preimage $\nabla P^{-1}(\{0\})$ of $0$ under the gradient map $\nabla P: \R^d \to \R^d$ has cardinality at least $m$.

To avoid degeneracy issues we now perform a perturbation argument.  Let $u \in \R^d$ be a sufficiently small vector, then we see that the function $x \mapsto P(x) - u \cdot x$ also has the property that at least one of the maximum or minimum of this function is attained only in the interior of $K_i$.  In particular, $\nabla P - u$ has at least one zero in the interior of $K_i$, and so the preimage $\nabla P^{-1}(\{u\})$ of $u$ under $\nabla P$ also has cardinality at least $m$.

Now take $u$ to be a sufficiently small \emph{generic} vector.  As the domain and range of the polynomial map $\nabla P$ have the same dimension, we see that $\nabla P^{-1}(\{u\})$ is finite for generic $u$ (even when $P$ is extended to the complex domain $\C^d$).  As such, we can invoke B\'ezout's theorem (considering the complex zeros) and conclude that the number of solutions is in fact bounded by $(D-1)^d = O(D^d)$ for generic $u$, as each component of $\nabla P$ has degree at most $D-1$.  The claim follows.
\end{proof}

A more complicated version of the above argument also works on general algebraic sets:

\begin{theorem}\label{pad}  Let $d \geq k \geq 0$, let $V$ be a $k$-dimensional real algebraic set in $\R^d$ of complexity at most $M$, let $W$ be another real algebraic set in $\R^d$ of complexity at most $M$, and let $P: \R^d \to \R$ be a polynomial of degree at most $D$, for some $D \geq 1$.  Then the set $\{ x \in V \backslash W: P(x) \neq 0 \}$ has $O_{M,d,k}(D^{k})$ connected components.
\end{theorem}

\begin{proof}  We allow all implied constants to depend on $M$, $d$, and $k$.  To give some additional flexibility we allow $P$ to have degree $O(D)$, rather than just having degree at most $D$.

We induct on the quantity $k+d$.  The case $k=0$ follows from Lemma \ref{cont}, so suppose that $k \geq 1$ and the claim has already been proven for smaller values of $k+d$.  By Lemma \ref{cont} we may assume $V$ to be irreducible.    We may assume that $P$ is non-vanishing on $V$, as the claim is trivial otherwise.

By Proposition \ref{complexity} and the induction hypothesis we may assume that $V^{\operatorname{smooth}}_\R$ is $k$-dimensional, and by Proposition \ref{most} and the induction hypothesis we can handle the contribution of the singular points of $V$.    Thus, by enlarging $W$ if necessary, we may assume without loss of generality that all real points of $V \backslash W$ are smooth.

In particular, all connected components of $\{ x \in V \backslash W: P(x) \neq 0 \}$ are now full dimensional in $V$.  This has the following consequence: if $P'$ is any multiple of $P$ that is still of degree $O(D)$ and still not vanishing identically on $V$, and $\{ x \in V \backslash W: P'(x) \neq 0 \}$ is known to have $O(D^k)$ components, then $\{ x \in V \backslash W: P(x) \neq 0 \}$ also has $O(D^k)$ components (since any component of the latter contains a component of the former).  As such, we now have the freedom to multiply $P$ at will by any polynomial of degree $O(D)$ that does not vanish identically on $V$.  In particular, by using polynomials on the bounded complexity set $W$ that do not vanish on $V$, we may assume that $P$ vanishes on $W$, at which point we can remove the role of $W$; thus we now assume that $P$ vanishes at all singular points of $V$, and our task is to show that $X=\{ x \in V: P(x) \neq 0 \}$ has $O(D^k)$ components.

As $V$ has complexity $O(1)$, it is not difficult to show (for instance, using an ultralimit compactness argument as in \cite{bgt}) that the (Zariski closure of the) normal bundle
$$NV := \{ (x,v) \in \R^d \times \R^d: x \in V^{\operatorname{smooth}}_\R; v \perp T_x V^{\operatorname{smooth}}_\R \}$$
of $V$ has complexity $O(1)$ as well.

Now we repeat the argument used to prove Theorem \ref{conn}.  It suffices to show for each cube $Q$ that the set
$$ \{ x \in V \cap Q: P(x) \neq 0 \} $$
has $O(D^k)$ components.  The components that touch the boundary of the cube can again be handled by the induction hypothesis, so we may restrict attention to the components that are in the interior of the cube.  By arguing as in the proof of Theorem \ref{conn}, it suffices to show that for generic $u \in \R^d$, the function $x \mapsto P(x)-u \cdot x$ on $V^{\operatorname{smooth}}_\R \cap Q$ has $O(D^k)$ interior critical points.

By the method of Lagrange multipliers, a point $x \in V^{\operatorname{smooth}}_\R$ is a critical point of $x \mapsto P(x)-u \cdot x$ if and only if the point $(x, \nabla P(x)-u)$ lies in $NV$.  On the other hand, for generic $u$, the set $\{ x \in V: (x, \nabla P(x)-u) \in NV \}$ (which is the intersection of the codimension $k$ set $NV$ with a generic translate of the graph of $\nabla P$ on the dimension $k$ set $V$) is finite.  As $\nabla P-u$ has degree $O(D)$ and $NV$ has complexity $O(1)$, an application of B\'ezout's theorem\footnote{More precisely, one can invoke here the affine B\'ezout inequality \cite{heintz}; see also \cite{rusek}, \cite{cgh}, \cite{schmid} for alternate proofs of this inequality.}  shows that this set has cardinality $O(D^k)$, and the claim follows.
\end{proof}

\end{document}